\documentclass[11pt]{amsart}
\usepackage[T1]{fontenc}
\usepackage{latexsym,amssymb}
\usepackage{epsfig}
\usepackage{amsfonts}
\usepackage{amscd}
\usepackage{amsmath}
\usepackage{amssymb}
\usepackage{amsthm}
\usepackage{bbm}
\usepackage{mathabx}
\usepackage{mathrsfs}
\usepackage{multicol}
\usepackage{siunitx}
\usepackage{textcomp}
\usepackage[utf8x]{inputenc} 
\usepackage{IEEEtrantools}
\usepackage{graphicx}
\usepackage{wrapfig}
\usepackage{paralist}
\usepackage{bbm}
\usepackage{xypic}
\usepackage{relsize}
\usepackage{verbatim}
\usepackage{cancel}
\usepackage{csquotes}
\usepackage{float}
\usepackage{cite}
\usepackage{tikz}
\usepackage{tikz-cd}
\usepackage{soul}
\usepackage{mathtools}
\usepackage{enumitem}
\usepackage{etoolbox}
\usepackage{changepage} 
\usepackage{hyperref}
\hypersetup{
    colorlinks=true,
    linkcolor=blue,
    filecolor=magenta,      
    urlcolor=cyan,
    pdftitle={Overleaf Example},
    pdfpagemode=FullScreen,
    }
\usepackage{comment}
\usepackage{caption}

\newtheorem{theorem}{Theorem}[section]

\newtheorem*{conj}{Conjecture}

\newtheorem{lemma}[theorem]{Lemma}
\newtheorem{proposition}[theorem]{Proposition}
\newtheorem{definition}{Definition}[section]

\newcommand{\C}{\mathbb{C}}

\newcommand{\N}{\mathbb{N}}
\newcommand{\Z}{\mathbb{Z}}

\newcommand{\GL}{\mathrm{GL}}

\newcommand{\SO}{\mathrm{SO}}
\newcommand{\SL}{\mathrm{SL}}
\newcommand{\SU}{\mathrm{SU}}
\newcommand{\semi}{\cdot_{\rtimes}}

\begin{document}

\title[Hyperbolic Integral Homology Spheres and $2I$ Representations]{Hyperbolic Integral Homology Spheres and Binary Icosahedral Representations}

\author{Maria Stuebner}
\email{ms6166@columbia.edu}

\vspace*{-1.5cm}
\maketitle

\begin{abstract}
This paper examines the representations of hyperbolic integral homology spheres into the binary icosahedral group $2I$. We specifically give a geometric meaning to $2I$ representations by relating them to Larsen's notion of \textit{quotient dimension}, which gives us a sense of the frequency of regular finite covers. Our main theorem shows that hyperbolic 3-manifolds can only have quotient dimension 2 or 3, and each case is obtained infinitely many times. More specifically, we show that those with no non-trivial $A_5$ representations have quotient dimension 3, and we find a family of hyperbolic 3-manifolds obtained by Dehn surgery on an infinite 2-bridge hyperbolic knot family with quotient dimension 2.  \\

\par
\end{abstract}


\setcounter{section}{0}

\section*{Introduction}

The classification and characterization of 3-manifolds is a key aim in low-dimensional topology. William Thurston's 1997 Geometrization Conjecture, proved by Perelman in 2003, places fundamental groups at the heart of the study of 3-manifolds, establishing the fundamental group $\pi_1$ as a primary tool for distinguishing between different closed 3-manifolds \cite{Thurston+1997}. Indeed, Perelman's proof tells us that fundamental groups determine closed, orientable 3-manifolds up to orientations of the prime factors and up to the indeterminacy arising from lens spaces \cite{perelman2002entropyformularicciflow} \cite{perelman2003ricciflowsurgerythreemanifolds} \cite{perelman2003finiteextinctiontimesolutions}. The Poincaré Conjecture, a special case of the Geometrization Conjecture, asserts that any simply connected, closed 3-manifold is homeomorphic to the three dimensional sphere $S^3$, characterizing $S^3$ as the unique, simply connected, closed 3-manifold. 

For 3-manifolds non-homeomorphic to $S^3$, it is useful to somehow quantify the non-triviality of their fundamental groups. The Geometrization Theorem implies that 3-manifolds have residually finite fundamental groups, so we can measure non-triviality by considering the representations to finite groups. However, there is no finite group $G$ such that every 3-manifold's fundamental group has a non-trivial homomorphism to $G$. We thus seek a different uniform measure of non-triviality via the following conjecture from Kirby's 1997 list of problems in low-dimensional topology \cite[3.15A]{Kirby}.

\begin{conj}[Kirby, 3.15A] If $Y$ is a closed, connected, 3-manifold other than $S^3$, then $\pi_1(Y)$ admits a non-trivial $\SU(2)-$representation.
\end{conj}

This conjecture highlights an important distinction between the 3-sphere, which has a trivial fundamental group, and other homology spheres, which, although they have the same homology as $S^3$, have more complex fundamental groups that can be mapped non-trivially into $\SU(2)$. Zentner showed in 2018 that the conjecture is true if $\SU(2)$ is replaced by $\SL_2 (\C)$ \cite{Zentner_2018}. In fact, the conjecture has been already established for several families of integer homology three-sphere, such as Seifert fibered homology three-spheres \cite{Fintushel1990InstantonHO} \cite{Sivek_2021}, branched double covers of non-trivial knots with determinant 1 \cite{Cornwell_2016} \cite{Zentner_2017}, 1/n-surgeries on non-trivial knots in $S^3$ \cite{kronheimer2003}, those that are filled by a Stein manifold which is not a homology ball \cite{Baldwin_2018}, or for splicings of knots in $S^3$  \cite{Zentner_2018}.

In this paper, we consider $2I$ representations, where $2I$ is the binary icosahedral group, a discrete subgroup of $\SU(2)$. The binary icosahedral group is the double cover of the alternating group $A_5$, where $A_5$ is isomorphic to the icosahedral group $I$, the rotational symmetry group of the regular icosahedron. Recall the short exact sequence of groups 
$$1 \rightarrow \Z/2 \rightarrow \SU(2) \xrightarrow {\pi} \SO(3) \rightarrow 1.$$ Just as $I$ is a discrete subgroup of $\SO(3)$, $2I$ is a discrete subgroup of $\SU(2)$. Thus, representations into $2I$ have the form
$$\pi_1 (Y) \rightarrow 2I \subset \SU(2).$$

This paper specifically focuses on giving a geometric meaning to $2I$ representations of 3-manifold fundamental groups by relating them to Larsen's notion of \textit{quotient dimension}, a measure used to link the algebraic properties of an object's fundamental group to the object's geometry. Larsen defines the quotient dimension of a finitely generated group $\Gamma$ as follows \cite[0.2]{Larsen}.

\begin{definition}[Larsen, 0.2] The \textbf{quotient dimension} $q_{\Gamma}$ is the minimal dimension of any linear algebraic group $G$ such that $G(\mathbb{C})$ contains an infinite quotient group of $\Gamma$. If no such $G$ exists, we say $q_{\Gamma} = \infty$. 
\end{definition}

In \cite{Larsen}, Larsen investigates the frequency of Hurwitz surfaces per genus, i.e. how often do Riemann surfaces with maximally-sized automorphism groups (of order $84(g-1)$) exist per genus $g$? Finding the quotient dimension of the fundamental groups of Hurwitz surfaces to be three, Larsen concludes that the genera of Hurwitz surfaces are roughly distributed like the \textit{perfect cubes}. 

In three-dimensions, quotient dimension represents something different: it will give us a sense of the distribution \textit{regular finite covers} of $X$. Focusing on the representations of hyperbolic integral homology spheres into the alternating group $A_5$, we arrive at the following main theorem.

\begin{theorem} \label{thm:main} Let $Y$ be a hyperbolic 3-manifold, and let $\Gamma = \pi_1(Y)$. Then, the quotient dimension of $\Gamma$ is either 2 or 3, and each case is obtained infinitely many times. More specifically,
\begin{enumerate}[start=1]
\item if $Y$ is obtained by $1/n$ surgery, where $n= \frac{10k-1}{7} \in \Z$, on the $K(40N+27, 20N+13)$ hyperbolic 2-bridge knot family ($N \in \mathbb{N}$), then $\Gamma$ has quotient dimension 2, and 
\item if $\Gamma$ does not map onto $A_5$, then $\Gamma$ has quotient dimension 3.
\end{enumerate}
\end{theorem}

\noindent Theorem 0.1 allows us to conclude that three-dimensional hyperbolic manifolds with regular finite coverings are distributed per genus as often as the \textit{perfect squares} and the \textit{perfect cubes}. Furthermore, there do indeed exist infinite families of homology spheres for each case.

\subsection{Outline} We begin by giving a short background on quotient dimension, specifically in the context of 2-manifolds, and we offer an interpretation of quotient dimension for hyperbolic 3-manifolds (\ref{sec:1}). We then focus on the hyperbolic integral homology sphere, a specific hyperbolic 3-manifold, and explore the following three possibilities for the quotient dimension of such a homology sphere: Can $q_{\Gamma} = 1$? (\ref{sec:2}) Can $q_{\Gamma} = 2$? (\ref{section:3}) And lastly, can $q_{\Gamma} = 3$ (\ref{section:4})? Propositions \ref{prop:1.2} and \ref{prop:2.3} prove the first part of Theorem \ref{thm:main} while we prove the second part, including points \textit{(1)} and \textit{(2)}, in the last two sections.


\section{Quotient dimension.} \label{sec:1}
Let $\Gamma$ be a finitely generated group. Quotient dimension gives us a way to \enquote{approximate} in some sense finitely generated groups like $\Gamma$ by linear algebraic groups. Larsen proves that the quotient dimension of $\Gamma$ singularly characterizes the convergence of $\Gamma$'s zeta function. Let us recall a few facts about zeta functions. For a finitely generated group $G$, we denote the index of a subgroup $H \unlhd G$ by $[G: H]$ and let
$$a_G^{\lhd} = | \{ H \unlhd G : [G: H] = n \}|.$$
The normal subgroup zeta function for $G$ is typically defined as 
$$\zeta_G (s) = \sum_{n = 1}^\infty a_G^{\lhd} (n) n^{-s} = \sum_{H \unlhd G} [G: H]^{-s},$$
\noindent where $H \unlhd G$ sums over all normal subgroups $H$ of finite index in $G$.

Similarly, Larsen defines his own zeta function, $Z_\Gamma(s)$. Let $H_\Gamma$ denote the set of isomorphism classes of finite homomorphic images of $\Gamma$ and let 
$$O(H_\Gamma) = \{ n \in \Z \ \  |  \ \ \exists H \unlhd \Gamma \text{ such that } [\Gamma: H] = n \},$$
i.e. the set of integers $n$ such that $\Gamma$ has at least one normal subgroup of index $n$. The set $O(H_\Gamma)$ is also the set of cardinalities of elements of $H_\Gamma$, i.e. of isomorphism classes of finite homomorphic images of $\Gamma$. He finally defines 
$$Z_\Gamma(s) = \sum_{n \in O(H_\Gamma)}n^{-s}.$$

We want a formula for the abscissa of convergence of $Z_{\Gamma}(s)$. Note that if the profinite completion of $\Gamma$ is finite, then the number of normal subgroups of finite index is finite, so $Z_{\Gamma}(s)$ becomes a finite sum rather than an infinite series. A finite sum of terms of the form $n^{-s}$ is an entire function in $s$, so the function converges everywhere in $\C$, and thus the abscissa of convergence is $\infty$. This case is trivial, so he only considers the case in which the profinite completion of $\Gamma$ is infinite. For the finitely generated group $\Gamma$ with infinite profinite completion, Larsen arrives at a formula for the abscissa of convergence of $Z_{\Gamma}(s)$ that relies on quotient dimension, shown below \cite[0.3]{Larsen}.
\begin{theorem}[Larsen, 0.3] Let $\Gamma$ be a finitely generated group with infinite profinite completion. The abscissa of convergence of $Z_{\Gamma}(s)$ is $1/q_{\Gamma}$.
\end{theorem}
\noindent Thus, quotient dimension not only gives us a measure of the approximate \enquote{size} of the finitely generated group $\Gamma$, but it also helps us understand how the zeta function of $\Gamma$ converges. This second application proves essential in the study of Hurwitz surfaces and their possible genus. 

\subsection{Hurwitz Surfaces} Let $\Gamma = \Delta(2, 3, 7)$, the triangle group, given by 
$$ \Delta(2, 3, 7) = \langle \ x, y, z \ | \ x^2 = y^3 = z^7 = 1 \ \rangle.$$
Let $G$ be a normal subgroup of finite index in $\Delta = \Delta(2, 3, 7)$, and let $\mathfrak{H}$ denote the upper half plane. The \textit{Hurwitz surfaces}, those with no more than $84(g-1)$ automorphisms, are exactly the complex curves $X$ given by 
$$X \simeq H \setminus \mathfrak{H},$$ 
for some finite index normal subgroup $H$ of the (2, 3, 7) triangle group $\Delta(2, 3, 7).$ Since $\Delta(2, 3, 7)$ is finitely generated, there are only finitely many such normal subgroups of a given index. Since $\text{Aut}(X) \simeq \Gamma / H$, we have $|\text{Aut}(X)| = [\Gamma : H]$, so in the case of Hurwitz surfaces, we have
$$n = [\Gamma: H] = 84(g(X)-1).$$

Since index is related to genus, there are only finitely many Hurwitz curves up to isomorphism in each genus. The distribution of such curves is not well known, and those found until now have been discovered experimentally. The current known sequence of genus for which a Hurwitz surface exists is 
\small
$$3, 7, 14, 17, 118, 129, 146, 385, 411, 474, 687, 769, 1009, 1025, 1459, 1537, 2091,...$$

\normalsize
\noindent The immediate question that arises is: how are the genera of these Hurwitz surfaces distributed? We want to use our formula for the zeta function of $\Gamma = \Delta(2, 3, 7)$. Recall we let $H_\Gamma$ denote the set of isomorphism classes of finite homomorphic images of $\Gamma$ and $O(H_\Gamma)$ the set of integers $n$ such that $\Gamma$ has at least one normal subgroup of index $n$. This set $O(H_\Gamma)$ is reduced to the set of natural numbers $84(g-1)$ for genus $g \in N$. Let $H_g$ be the finite set of isomorphism classes of Hurwitz curves in genus $g$. Then, we rewrite 
$$Z_{\Gamma}(s) = \sum_{n \in O(H_{\Gamma})} n^{-s} = \sum_{X \in H_g} \frac{1}{84(g(X)-1)^s}$$
where $\sum_{X \in H_g} \frac{1}{84(g(X)-1)^s}$ has the same convergence properties as $\sum_{X \in H_g} \frac{1}{g(X)^s}.$ Thus, the genus $g$ of Hurwitz surfaces up to isomorphism satisfy
$$Z_{\Gamma} (s) = \sum_{g \in G} g^{-s},$$
where $G$ denotes the set of integers $g \geq 2$ such that there exists at least one compact Riemann surface of genus $g$ with automorphism group of order $84(g-1)$, i.e. at least one Hurwitz surface. From Theorem 1.1, the abscissa of convergence of $Z_{\Gamma} (s)$ above is $1/q_{\Gamma}$, and from \cite[2.4]{Larsen}, $q_{\Gamma} = 3$ for $\Gamma = \Delta(2, 3, 7)$, so the series $\sum_{g \in H} g^{-s}$ converges absolutely for $\mathcal{R}(s) > 1/3$. Thus, the genus of Hurwitz surfaces are distributed roughly like the cubes. 

Moreover, Larsen proves that the triangle group $\Gamma = \Delta (a, b, c)$ for $a, b, c$ positive integers has quotient dimension $\leq 3$, and more specifically, he gives three examples for each possible value of quotient dimension:
\begin{enumerate}
\item $q_{\Gamma} = 3$, where $\Gamma = \Delta (2, 3, 7),$ 
\item $q_{\Gamma} = 2$, where $\Gamma = \Delta (2, 3, 8),$
\item $q_{\Gamma} = 1$, where $\Gamma = \Delta (2, 3, 12)$.
\end{enumerate}
What happens, however, when we leave two-dimensional space and move into three-dimensional space? More specifically, we ask two questions: 
\begin{enumerate}
\item what does quotient dimension mean in 3 dimensions? and 
\item what possible values of quotient dimension exist in 3 dimensions? 
\end{enumerate}

We start with the first. Let $Y$ be a hyperbolic 3-manifold, and $\Gamma = \pi_1(Y)$. In our original manipulation of the zeta-function, we used the fact the indices of the normal subgroups of $\Gamma = \Delta (2, 3, 7)$ were inherently related to the genus of the Hurwitz surfaces. In three-dimensional space, the normal subgroups of a space's fundamental group correspond to its set of regular covers, we find that the quotient dimension in three-dimensional space tells us about the distribution of \textit{regular covers} of $Y$. There is a one to one correspondence between the covering spaces of a Riemannian manifold $M$ and the subgroups of the group action $\Gamma < \text{Isom}(M)$ \cite[1.38]{hatcher2002algebraic}. A covering of type $M \rightarrow M/\Gamma$ is called \textit{regular}, and the corresponding fundamental group $\pi_1 (M)$ is normal in $\pi_1 (M/\Gamma) = \Gamma$, \cite[1.39]{hatcher2002algebraic} so we also have:
$$\{ \text{\textbf{regular} coverings of }M    \}  \leftrightarrow   \{ \text{\textbf{normal} subgroups of } \Gamma \}.$$

A covering space is regular if for every $x \in X$ and pair of lifts $\tilde{x}, \tilde{x}'$, there is a deck transformation (isomorphism $\tilde{X} \rightarrow \tilde{X}$) taking $\tilde{x}$ to $\tilde{x}'$ (see  \cite[Pg. 70]{hatcher2002algebraic} for more details). Recall that for a finitely generated group $\Gamma$, the quotient dimension $q_{\Gamma}$ tells us how quickly the sum $Z_\Gamma(s) = \sum_{n \in O(H_\Gamma)}n^{-s}$, where $O(H_\Gamma)$ is the set of integers $n$ such that $\Gamma$ has at least one normal subgroup of index $n$, converges. Since $\Gamma$ is also the fundamental group of $M$, the indices of the normal subgroups of  $\Gamma$ are distributed similar to the powers of $q_{\Gamma}$, and thus so are the locally symmetric coverings of $M$.

We turn now to the second question. In 3-dimensions, the notion of Hurwitz surfaces no longer exists; however, similar to triangle groups, the isometries of hyperbolic space are contained in $PSL_2(\mathbb{C})$, so we do know that the quotient dimension of hyperbolic 3-manifolds $q_{\Gamma} \leq 3$. We thus limit our discussion of three-dimensional quotient dimension to three-dimensional hyperbolic space, noting the following formalized proposition.

\begin{proposition} \label{prop:1.2} Let $Y$ be hyperbolic 3-manifold, and $\Gamma = \pi_1(Y)$. Then, 
$$q_{\Gamma} \leq 3.$$
\end{proposition}

\begin{proof}
Similar to the case of triangle groups, the isometries of hyperbolic space are contained in $\text{PSL}_2(\C)$, so the quotient dimension of hyperbolic 3-manifolds can be no more than 3. 
\end{proof}


\section{Quotient dimension cannot be 1.} \label{sec:2}
In the case of quotient dimension one, let us first define a few properties about linear algebraic groups. Let $G$ be a linear algebraic group over $\mathbb{C}$ of minimal dimension among those whose complex points contain an infinite quotient group of $\Gamma$. If $N$ is a normal subgroup of $G$ of positive dimension, then the image $\Gamma'$ of $\Gamma$ in the quotient $(G/N)(\mathbb{C})$ must be finite. The preimage of $\Gamma'$ in $G$ contains $\Gamma$ and must therefore have the same dimension as $G$, so $\text{dim} \, \, N = \text{dim} \, \, G$. Consider the definition of \textit{pseudo-simple} from \cite[1.1]{Larsen}. 

\begin{definition}[Larsen, 1.1] A linear algebraic group $G$ is pseudo-simple when every normal linear subgroup $N$ of $G$ is finite or of finite index in $G$.
\end{definition}

\subsection{Some Normal Linear Subgroups of $G$}
Every linear algebraic group $G$ has a unique irreducible component $G^o$ containing the identity element $e$, and $G^o$ is a closed normal subgroup of finite index in $G$ \cite[2.1]{Macdonald_Carter_MacDonald_Segal_1995}. The radical $R (G)$ of $G$ is the unique maximal, closed, connected, solvable, normal subgroup of $G$. The unipotent radical $R_u (G)$ is the unique maximal, closed, connected, unipotent, normal subgroup of $G$. Note that $R_u (G) \subset R(G)$. The derived unipotent radical $R_u^{der} (G)$ is the derived, or commutator, subgroup of $R_u (G)$. We thus have the following filtration of $G$:
$$\{0\} \subset R_u^{der} (G) \subset R_u (G) \subset R (G) \subset G^o \subset G.$$

\noindent  If the radical $R(G) = \{ 0 \},$ the group $G$ is said to be \textit{semisimple}. If the unipotent radical $R_u (G) = \{ 0 \},$ $G$ is said to be \textit{reductive} \cite[Interlude]{Macdonald_Carter_MacDonald_Segal_1995}. A solvable reductive group is said to be \textit{multiplicative}. If $R_u^{der} (G) = \{ 0 \},$ then $R_u (G)$ is commutative. Consider now a lemma from Larsen \cite[1.2]{Larsen}.

\begin{lemma}[Larsen, 1.2] If $G$ is pseudo-simple, the identity component $G^{0}$ is either semi-simple, multiplicative, or commutative and unipotent. 
\end{lemma}

\noindent This lemma follows from the fact that as characteristic subgroups of $G^o$, the radical $R(G^o)$, unipotent radical $R_u(G^o),$ and derived unipotent radical $R_u(G^o)^{der}$ are normal subgroups. By Definition 2.1, the filtration
$${0} \subset R_u(G^o)^{der} \subset R_u(G^o) \subset R(G^o) \subset G^o$$
has one non-trivial step, i.e. $\{0\} \subset G^o$. The three cases, (1) semi-simple, (2) multiplicative, and (3) commutative and unipotent correspond to the quotients arising from the three possible non-trivial steps:
\begin{enumerate}
\item $R_u(G^o)^{der} = \{ 0 \} \hfill \{0\} \subset R_u(G^o) = R(G^o) = G^o$
\item $R_u(G^o)^{der} = R_u(G^o) = \{ 0 \} \hfill \{0\} \subset R(G^o) = G^o$
\item $R_u(G^o)^{der} = R_u(G^o) = R(G^o) = \{ 0 \} \hfill \{0\} \subset G^o$
\end{enumerate}

\noindent The three possibilities for the identity component $G^{0}$ in Lemma 2.1 correspond to three types for $G$. The group $G$ is said to be of \textbf{non-abelian}, \textbf{toral}, or \textbf{vector} if the identity component $G^{0}$ is semi-simple, multiplicative, or commutative and unipotent, respectively. For pseudo-simple groups of toral type, Larsen defines the following proposition \cite[1.3]{Larsen}.

\begin{proposition}[Larsen, 1.3] If $G$ is a pseudo-simple linear algebraic group of toral type such that $G(\C)$ contains an infinite homomorphic image of $\Gamma$, then there exists a pseudo-simple group $G'$ of vector type with $\dim G' \leq \dim G$ such that $G' (\C)$ also contains an infinite quotient of $\Gamma$. 
\end{proposition}

Using Proposition 2.2, we can now prove that the quotient dimension of a hyperbolic homology sphere cannot be 1. 

\begin{proposition} \label{prop:2.3} Let $\Gamma = \pi_1(Y)$, where $Y$ is a hyperbolic homology sphere. $\Gamma$ has quotient dimension $\geq 2$.
\end{proposition}

\begin{proof} Assume to the contrary that $\Gamma$ has quotient dimension 1. By Proposition 2.2, we can assume that $\Gamma$ has an infinite quotient Zariski-dense in $G(\C)$, an extension of a finite group by $\C$. $G/Z_{G} (G^{o})$, the quotient of $G$ by the centralizer of the identity component $G^o$, acts faithfully on $\C$, so it is finite cyclic, say of order $d$. Then, $\Gamma = \pi_1 (Y)$ maps onto $\Z/d\Z$ via some map 
$$f: \Gamma \rightarrow \Z/d\Z.$$

\noindent Any homomorphism $h: \pi_1(Y) \rightarrow G,$ where $G$ is any abelian group, induces a homomorphism $h': \pi_1(Y)^{ab} \rightarrow G$. Since $Y$ is path-connected, we have that by Hurewicz's theorem, $H_1(Y)$ is the abelianization of $\pi_1(Y)$, so the surjective map $f: \Gamma \rightarrow \Z/d\Z$ induces the surjective homomorphism
$$\tilde{f}: \Gamma^{ab} = H_1(Y) \rightarrow \Z/d\Z.$$
However, $H_1(Y) = 0$ since $Y$ is a homology sphere, so it cannot map surjectively onto $\Z/d\Z$. We thus have a contradiction, so $q_{\Gamma} \geq 2$. 
\end{proof}

Together with Proposition 1.2, Proposition 2.3 concludes the proof of the first part of Theorem \ref{thm:main}.


\section{Quotient dimension can sometimes be 2.} \label{section:3}
In the case of quotient dimension 2, it is not directly clear whether or not a 3-manifold can even have $q_{\Gamma} = 2$. However, using GAP  \cite{GAP} as a computational tool and following in footsteps of Larsen's proof for the quotient dimension of $\Gamma = \Delta(2, 3, 8)$ \cite[2.4]{Larsen}, we find several families of homology spheres with quotient dimension 2, one of which is presented in this section, indicating that there do in fact exist 3-manifolds with quotient dimension 2.

In order to show that the quotient dimension of $\Gamma$ is some $n$, we need to show that $n$ is the minimal dimension of any linear algebraic group $G$ such that $G(\C)$, the group of $\C$-valued points of $G$, contains an infinite quotient group of $\Gamma$. In the case of quotient dimension 2, we note that it is sufficient to find a single linear algebraic group $G$ of dimension 2 such that $G(\C)$ contains an infinite quotient group of $\Gamma$. Then, by Theorem \ref{thm:main}(1), $\Gamma$ must have quotient dimension 2 (since 2 is the minimal possible dimension).  

\subsection{Constructing $G(\C)$} Consider the linear algebraic group $H$ that acts faithfully on $\C^2$ via a linear and polynomial action. Then, there exists a linear algebraic group $G$ of dimension 2 such that the semi-direct product $\C^2 \rtimes H = G(\C)$. We can thus restrict our focus to linear algebraic groups that act faithfully on $\C^2$ via a linear and polynomial action. 

We take this approach from \cite[2.4]{Larsen}. For example, to prove that $\Gamma_{2, 3, 8}$ has quotient dimension $2$, Larsen finds a commutator of infinite order in $\C^2 \rtimes \text{GL}_2 (\mathbb{F}_3)$, where there exists a surjective homomorphism $\Gamma_{2, 3, 8} \rightarrow  \text{GL}_2 (\mathbb{F}_3)$. Similarly, noting the surjective homomorphism $\Gamma_{2, 3, 12} \rightarrow \Z/6 \Z$, he proves that $q_{\Gamma_{2, 3, 12}} = 1$ by finding a commutator of infinite order in $\C \rtimes \Z/6 \Z$. It is in manner that we are able to prove that an infinite knot family has quotient dimension 2 (Theorem \ref{thm:main}(1)). Before we continue with the proof,  we make a few notes, recalling some useful definitions, properties, and facts. 

\subsubsection{Semi-direct Product} Since we will do a bit of work with semi-direct products, let us recall a few properties. Let $N$ and $H$ be groups, with $H$ acting on $N$, i.e. there exists a group homomorphism $\rho : H \rightarrow \text{Aut} (N)$. The external semi-direct product $G = N \rtimes H$ is equivalent as a set to the cartesian product $N \times H$, with the product rule
$$(a, b) \semi (a', b') = (a (\rho(b)(a')), bb').$$
We can rewrite the action $\rho(b) a' = b \cdot a'$, so that we then have
$$(a, b) \semi (a', b') = (a (b \cdot a'), bb').$$
Since we mostly work with additive abelian groups, we use the following product rule: for $N$ abelian with the binary operation addition, we have
$$ (a, b) \semi (a', b') = (a + (b \cdot a')), bb').$$
See \cite{Conrad} for more details on semi-direct products.

\subsubsection{Lifting $A_5$ to the Binary Icosahedral Group $2I$}
Suppose $\Gamma$ maps surjectively onto $A_5$ (we handle the case in which $\Gamma$ does \textit{not} map onto the alternating group $A_5$ in Section 4). Unfortunately, $A_5$ does not act on $\C^2$ as it cannot be embedded into $\GL(2, \C)$, so we cannot construct the extension $\C^2 \rtimes A_5$. However, the \textit{double cover} of $A_5$ (to which the surjective homomorphism $\Gamma \rightarrow A_5$ lifts), the binary icosahedral group $2I$, \textit{does} act on $\C^2$ as a subgroup of $\SU(2) \subset \GL(2, \C)$. The group has presentation
$$2I = \langle x, y \ | \ (xy)^2 = x^3 = y^5 = 1 \rangle,$$
\noindent and order 120. $2I$ acts freely on $S^3$ since it is a finite subgroup of the group $SU(2) \cong S^3$ which can be thought of as the unit quaternions $\text{Spin}(3)$. The binary icosahedral subgroup in $\text{Spin}(3)$ is the inverse image of $A_5$ under the double covering map. Since $A_5$ is simple, $H_1(A_5; \Z) = 0$. Also, we have that $H_2(A_5; \Z) = \Z/2$. Thus, it is the unique maximal central extension of $A_5$, which is simple, so it fits into the short exact sequence
$$1 \rightarrow \Z/2 \rightarrow 2I \rightarrow A_5 \rightarrow 1.$$
This sequence does not split, so $2I$ is not a semi-direct product of $\Z/2$ by $A_5$. The center of $2I$ is the subgroup $\Z/2$, so the inner automorphism group is isomorphic to $I$. The full automorphism group is isomorphic to the symmetric group $S_5$, just as $I \cong A_5$. Thus, any automorphism of $A_5$ lifts to an automorphism of $2I$ since the lift of the generators of $I$ are generators of $2I$. 

\subsubsection{Quaternions}
The generators of the binary icosahedral group are given in terms of unit quaternions \cite{coxeter2013generators} \cite{quaternion}. The quaternion number system extends the complex numbers $\C$, with basis $\{ 1, i \}$, to $\C^2$ by adding two new basis elements $j$ and $k$, given by the equation
$$i^2 = j^2 = k^2 = ijk = -1.$$

Indeed, $2I = \langle x, y \ | \ (xy)^2 = x^3 = y^5 = 1 \rangle,$ where 
$$x = \frac{1+i+j+k}{2} \hspace{0.5in} y = \frac{\varphi + i/\varphi +j}{2},$$
and where $\varphi$ is the golden ratio given by $\varphi = \frac{1 +\sqrt{5}}{2}$.

\subsubsection{2-Bridge Knots}
Let $K = K(p,q)$ be a hyperbolic 2-bridge knot. For $i \in \N$ with $i \leq p$, let $e_i = (-1)^{\lfloor \frac{iq}{p} \rfloor}$ and let $\sigma = \sum_{i=1}^{p-1} e_i$. Then, the fundamental group $\pi_1(S^3 - K)$ can be presented as a group
$$\Gamma = \langle \, x, y \, \, | \, \, wxw^{-1}y^{-1} \rangle,$$
where $w = x^{e_1} y^{e_2} \cdots x^{e_{p-2}} y^{e_{p-1}}$.
The meridian is given by $x$ and the Seifert longitude is given by $x^{-2 \sigma}w_{*}w$ \cite[Proposition 1]{10.1112/plms/s3-24.2.217}, where
$$w_{*} = y^{e_1} x^{e_2} \cdots y^{e_{p-2}} x^{e_{p-1}}.$$

We work with Dehn surgeries on knots, whose fundamental groups are closely related to knot groups. Let $\mu$ be the meridian of the knot and $\lambda$ the Seifert longitude. The fundamental group $\pi_1(S_{p/q}^3(K)$) is isomorphic to $\pi_1(S^3 - K)/\langle \mu^p \lambda^q \rangle$, where $\langle \mu^p \lambda^q \rangle$ refers to the subgroup generated by $\mu^p \lambda^q$. In our case, we are interested in 1/n, or integral, surgeries. Let $\Gamma = \pi_1 (S_{1/n}^3)$. Then, the presentation of $\Gamma$ is given by
$$\Gamma = \langle \, x, y \, \, | \, \, wxw^{-1}y^{-1}, \ x(x^{-2 \sigma}w_{*}w)^{n} \, \rangle.$$

With 3.1.1-4, we are now ready to prove Theorem \ref{thm:main}(1).

\vspace{0.1in}

\noindent \textbf{Theorem \ref{thm:main}(1).} \textit{Let $S_{1/n}^3 = S^3 - K(p,q)$ be the $1/n$ surgery associated to the hyperbolic 2-bridge knot $K(p,q)$ for $n= \frac{10k-1}{7} \in \Z$. Let $\Gamma = \pi_1 (S_{1/n}^3)$. If $(p, q) = (40N+27, 20N+13)$ for $N \in \N$, then $\Gamma$ has quotient dimension 2.}

\vspace{0.1in}

\begin{proof}
We note first the surjective map $f: \Gamma \rightarrow 2I$ such that 
\begin{equation*}
x \mapsto -b^2 \hspace{1in} y \mapsto a^2b^2a,
\end{equation*}
checked computationally with GAP \cite{GAP}. We define $F: \Gamma \rightarrow \C^2 \rtimes 2I$ as
\begin{equation*}
x \mapsto (\vec{u}, f(x)) \hspace{1in} y \mapsto (\vec{v}, f(y))
\end{equation*}
for some $\vec{u}, \vec{v} \in \C^2$. Let $\rho: 2I \rightarrow \C^2$ denote the non-trivial representation induced by the natural action of $2I$ on $\C^2$ via multiplication. \\

\noindent We want to find $\vec{u}$ and $\vec{v}$ such this mapping is a homomorphism, so we need to check that $F(r_i) = \text{Id}_{\C^2 \rtimes 2I} = (\vec{0}, 1)$ for all the relators $r_i$ of $\Gamma$. Recall from 3.1.4 that the relators of $\Gamma$ are
\begin{eqnarray*}
r_1 &=& wxw^{-1}y^{-1} = 1\\
r_2 &=& x(x^{-2 \sigma}w_{*}w)^{n} = 1,
\end{eqnarray*}
so we want to find $F$ such that
\begin{eqnarray*}
&F(r_1) = F(1) = (\vec{0}, 1)&\\
&F(r_2) = F(1) = (\vec{0}, 1).& 
\end{eqnarray*}

\noindent Since $F(r_1) = F(wxw^{-1}y^{-1}) = (a\vec{u}+ b\vec{v}, f(wxw^{-1}y^{-1}))$ for some $a, b \in \C$, and similarly $F(r_r) = F(x(x^{-2 \sigma}w_{*}w)^{n}) = (c\vec{u}+ d\vec{v}, f(x(x^{-2 \sigma}w_{*}w)^{n}))$ for some $c, d \in \C$, we note that for $F$ to be a homomorphism, we must have (1) $f(r_1) = f(r_2) = 1$, and (2) $a\vec{u}+ b\vec{v} = c\vec{u}+ d\vec{v} = \vec{0}$, i.e. for $ad = bc$. Indeed, in defining $f$ above, condition (2) is already satisfied (via GAP computations \cite{GAP}), so we limit ourselves to showing that (1) is satisfied. \\

\noindent We start by proving that $f(r_1) = f(w x w^{-1} y^{-1}) = 1$. We can compute $w$ directly, noting first that since $2(20N+13) < 40N+27$, we have
\begin{eqnarray*}
e_1 &=& (-1)^{\lfloor \frac{20N+13}{40N+27} \rfloor} = (-1)^0 =1\\
e_2 &=& (-1)^{\lfloor \frac{40N+26}{40N+27} \rfloor} = (-1)^0 = 1\\
e_3 &=& (-1)^{\lfloor \frac{3(20N+13)}{40N+27} \rfloor} = (-1)^1 = -1\\
e_4 &=& (-1)^{\lfloor \frac{4(20N+13)}{40N+27} \rfloor} = (-1)^1 = -1\\
e_5 &=& (-1)^{\lfloor \frac{5(20N+13)}{40N+27} \rfloor} = (-1)^0 = 1\\
&\vdots&\\
e_{40N+24= p-3} &=& (-1)^{\lfloor \frac{(10N+6) \cdot 4 \cdot q}{p} \rfloor} = (-1)^1 = -1\\
e_{40N+25= p-2}  &=& 1 \\
e_{40N+26 = p-1} &=& 1. 
\end{eqnarray*}
The first $p-3$ terms cancel, so we have that $\sigma = \sum_{i=1}^{p-1} e_i = e_{p-2} + e_{p-1} = 2$. It follows from the definitions in 3.1.4 that
$$w = xyx^{-1}y^{-1} \cdots xyx^{-1}y^{-1}xy \hspace{0.5in} w_{*} = yxy^{-1}x^{-1} \cdots yxy^{-1}x^{-1}yx.$$
Since $p-3 = 40N+24= 4(10N+6),$ we note that we can rewrite $w$ as
\begin{equation*}
w = (xyx^{-1}y^{-1})^{10N} (xyx^{-1}y^{-1})^{6} xy.
\end{equation*}
It will useful for us to write $w$ in this manner. Let $M = x y x^{-1} y^{-1}$. Then,
$$w = M^{10N} M^6 xy.$$
We then rewrite the relator $r_1$ as 
\begin{eqnarray*}
r_1 &=&  M^{10N} M^6 xy \cdot x \cdot (M^{10N} M^6 xy)^{-1} \cdot y^{-1}\\
 &=&  M^{10N} M^6 xyxy^{-1}x^{-1} M^{-6} M^{-10N} y^{-1}\\
  &=&  M^{10N} M^6 x(xyx^{-1}y^{-1})^{-1} M^{-6} M^{-10N} y^{-1}\\
  &=&  M^{10N} M^6 x M^{-1} M^{-6} M^{-10N} y^{-1}\\
    &=&  M^{10N} M^6 x M^{-7} M^{-10N} y^{-1}.
\end{eqnarray*}
Note that $f(M)$ has order $10$, i.e. $f(M)^{10} = 1$, checked via GAP \cite{GAP}:

\begin{figure}[H]
\centering
\vspace{-0.15in}
\includegraphics[scale=0.7]{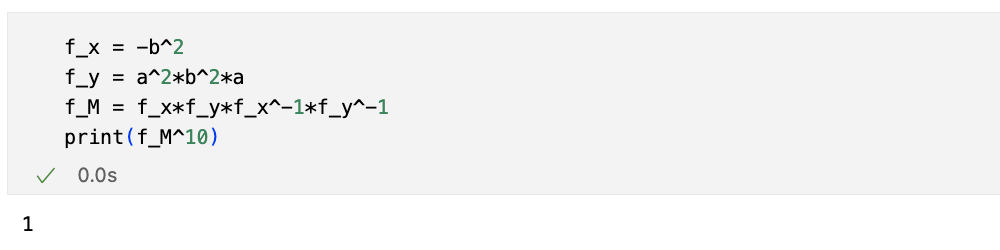}
\vspace{-0.2in}
\end{figure}

\noindent Thus, $f(r_1)$ is reduced to 
\begin{eqnarray*}
f(r_1) &=& f(M^{10N}) f(M^6) f(x) f(M^{-7}) f(M^{-10N}) f(y^{-1}) \\
&=& f(M^6) f(x) f(M^{-7}) f(y^{-1}).
\end{eqnarray*}

\noindent This computation is important because it tells us that although $r_1$ depends on the $n$ from $K(40N+27, 20N+13)$, $f(r_1)$ \textit{does not depend on N}, so the following result holds for the \textit{infinite} knot family. Noting that $M^6 x M^{-7} y^{-1} = 1$ (computed with GAP \cite{GAP}, see below), we conclude that $f$ maps the first relator to the identity, satisfying the half of condition (1).

\begin{figure}[H]
\centering
\vspace{-0.2in}
\includegraphics[scale=0.7]{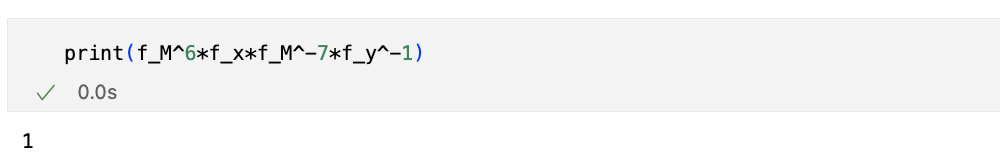}
\vspace{-0.2in}
\end{figure}

\noindent As noted above, $f$ is picked such that condition $(2)$ is satisfied, but we do note that this is only possible to conclude for the entire knot family because $w$ does not depend on $N$, so the semi-direct product does not either. Indeed, using GAP \cite{GAP} to compute the total semi-direct product, we have
$$F(r_1) = F(M^6 x M^{-7} y^{-1}) = F(M^6) \semi F(x) \semi F(M^{-7}) \semi F(y^{-1}) = (\vec{0},1).$$

\noindent We now show that $f(r_2) = 1$. First note that 
$$f(M)^5 = f(xyx^{-1}y^{-1})^{5} = -1 \hspace{0.35in} f(M) = f(xyx^{-1}y^{-1})^{10} = 1,$$
and similarly, for $\tilde{M} = y^{-1}x^{-1}yx$,
$$f(\tilde{M})^5 = f(y^{-1}x^{-1}yx)^{5} = -1 \hspace{0.3in} f(\tilde{M}^{10}) = f(y^{-1}x^{-1}yx)^{10} = 1.$$
Using the above two calculations, we have \begin{eqnarray*}
f(w_* w) &=& f(yx(y^{-1}x^{-1}yx)^{10N}(y^{-1}x^{-1}yx)^{6}(xyx^{-1}y^{-1})^{6} (xyx^{-1}y^{-1})^{10N}xy)\\
 &=& f(yx(y^{-1}x^{-1}yx)^{6}(xyx^{-1}y^{-1})^{6} xy)\\
  &=& f(yx(y^{-1}x^{-1}yx)) f(y^{-1}x^{-1}yx)^{5} f(xyx^{-1}y^{-1})^5 f(xyx^{-1}y^{-1}) (xy) \\
   &=& f(yx(y^{-1}x^{-1}yx)) \cdot (-1) \cdot (-1) \cdot f(xyx^{-1}y^{-1}) (xy)\\
      &=& f(yx \tilde{M} Mxy),
\end{eqnarray*}
so $f(w_{*}w)$ does not depend on $N$. Moreover, GAP \cite{GAP} tells us that 
$$f(w_* w) = f(yx \tilde{M} Mxy) = f(x) \ \ \ \ \ \text{(see below)}.$$
\begin{figure}[H]
\centering
\vspace{-0.05in}
\includegraphics[scale=0.7]{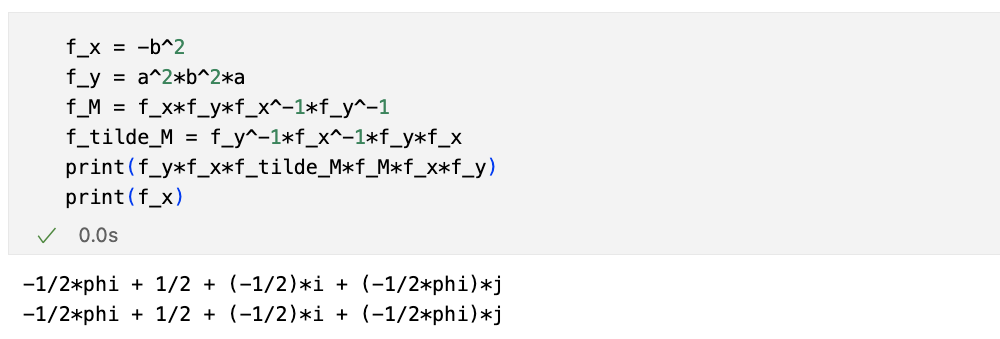}
\vspace{-0.15in}
\end{figure}

\noindent Recalling that $\sigma = 2$, we can then rewrite the second relator as
\begin{eqnarray*}
f(r_2) &=& f(x (x^{-2\sigma} w_{*}w)^n) \\
&=& f(x (x^{-4} x)^n) \\
&=& f(x (x^{-3})^n).
\end{eqnarray*}
Since $f(x)$ has order 10 (checked with GAP), $f(x^{-3}) = f(x^7)$. Recall by hypothesis that $n= \frac{10k-1}{7} \in \Z$ for $k\in \Z$. Indeed, then we have
$$f(r_2) = f(x (x^{7})^n) =  f(x^{1+7n}) = f(x)^{1+7n} = f(x)^{10k} = (f(x)^10)^k = 1.$$

\noindent Similar to before, since $f(r_2)$ does not depend on $n$, the semi-direct product does not either. Using GAP \cite{GAP}, we compute the total semi-direct product 
$$F(r_2) = F(x(x^{-4}w_{*}w)^n) = F(x) \semi F(x^{-4}w_{*}w)^n = (\vec{0},1).$$

\vspace{0.2in}

\noindent Thus, we have seen that $F(r_1)$ and $F(r_2)$ both map to the identity in $\C^2 \rtimes 2I$ regardless of how $\vec{u}$ and $\vec{v}$ are chosen in $\C^2$. Moreover, for any two $\vec{u}, \vec{v} \in C^2$ such that at least one is non-zero, $F$ is non-trivial. \\

\noindent Let $\vec{u}\in \C^2$ and $\vec{v} =\vec{0} \in \C^2$. Consider again $F: \Gamma \rightarrow \C^2 \rtimes 2I$, where now
$$x \mapsto (\vec{u}, -b^2) \hspace{1in} y \mapsto (0, a^2b^2a).$$
We show that the commutator of $x^5$ with $y$ has infinite order under $F$.

\newpage
\noindent We begin by noting that 

\vspace{-.2in}
\begin{eqnarray*}
F(x^5) &=& F(x) \semi F(x) \semi F(x) \semi F(x) \semi F(x) \\
&=& (\vec{u}, -b^2) \semi (\vec{u}, -b^2) \semi (\vec{u}, -b^2) \semi (\vec{u}, -b^2) \semi (\vec{u}, -b^2) \\
&=& (\vec{u}(1 - b^2 + b^4 - b^6 +b^8), -b^{10})  \\
&=& (\vec{u}(1 - b^2 + b^4 - b^6 +b^8), -1),
\end{eqnarray*}

\noindent where $1 - b^2 + b^4 - b^6 +b^8 = 1 + (\varphi - 2) \cdot i + (-\varphi + 1) \cdot j \not= 0.$ Then, 

\vspace{-.2in}
\begin{eqnarray*}
F(x^5yx^{-5}y^{-1}) &=& F(x^5) \semi F(y) \semi F(x^{-5}) \semi F(y^{-1}) \\
&=& (\vec{u}(1 - b^2 + b^4 - b^6 +b^8), -1) \semi (0, a^2b^2a) \\
&& \ \semi (\vec{u}(1 - b^2 + b^4 - b^6 +b^8), -1) \semi (0, a^{-1}b^{-2}a^{-2}) \\
&=& (\vec{u}(1 - b^2 + b^4 - b^6 +b^8), -1) \semi (0, a^2b^2a) \\
&& \ \ \ \ \ \ \ \ \ \ \semi (\vec{u}(1 - b^2 + b^4 - b^6 +b^8), -a^{-1}b^{-2}a^{-2}) \\
&=& (\vec{u}(1 - b^2 + b^4 - b^6 +b^8), -1) \\
&& \ \ \ \ \ \ \ \ \ \ \semi (a^2b^2a \cdot \vec{u}(1 - b^2 + b^4 - b^6 +b^8), -1) \\
&=& (\vec{u}(1 - b^2 + b^4 - b^6 +b^8) - a^2b^2a \cdot \vec{u}(1 - b^2 + b^4 - b^6 +b^8), 1) \\
&=& (\vec{u}(1 - b^2 + b^4 - b^6 +b^8)(1-a^2b^2a), 1),
\end{eqnarray*}

\noindent where $(1 - b^2 + b^4 - b^6 +b^8)(1-a^2b^2a) = \varphi + (-\varphi + 1) \cdot j + k \not= 0$. Thus, 
$$F(x^5yx^{-5}y^{-1})^n = (n \cdot \vec{u}(1 - b^2 + b^4 - b^6 +b^8)(1-a^2b^2a), 1).$$
Since the first coordinate of the above will never be zero for any $n \in \N$, the commutator of $x^5$ and $y$ is of infinite order in $\C^2$. \\

\noindent Thus, $\Gamma$ has quotient dimension $2$.
\end{proof}


\section{Quotient dimension can sometimes be 3.} \label{section:4}
Below, we state and prove a theorem that gives us a sufficient condition for the quotient dimension of a hyperbolic homology sphere to be 3. \\

\noindent \textbf{Theorem \ref{thm:main}(2).} Let $Y$ be a hyperbolic homology sphere with $\Gamma = \pi_1 (Y)$. If $\Gamma$ does not map onto $A_5$, then $\Gamma$ has quotient dimension 3.

\begin{proof}
Assume to the contrary that the quotient dimension is less than 3. We can then assume without loss of generality that the quotient dimension is 2, so by Proposition 2.2, there exists a 2-dimensional pseudo-simple group $G$ of vector type such that $G(\mathbb{C})$ contains a quotient of $\Gamma$ of vector type. 

Then, the quotient group $G/Z_G(G^o)$ is a homomorphic image of $\Gamma$, and it acts faithfully on $\mathbb{C}^2$. By the classification of finite subgroups of $PGL_2(\mathbb{C}),$ every finite subgroup of $GL_2(\mathbb{C})$ which is not solvable maps onto the alternating group $A_5$ \cite{CJPS_1998__16_} \cite{zbMATH01466287}. However, $\Gamma$ does not map onto $A_5$ by hypothesis, so $z=1$ for all $z \in G/Z_G(G^o)$, so the quotient group $G/Z_G(G^o)$ is trivial. $G/Z_G(G^o) = \{ 1 \}$ implies that $G \cong Z_G (G^o),$ so $G^o \subseteq Z(G)$. 

Consider the following short exact sequence:
$$1 \rightarrow G^o \rightarrow G \rightarrow \pi_0 (G) \rightarrow 1.$$
Since $G^o \subseteq Z(G)$, it follows that $G$ is a central extension of $\pi_0 (G) $ by $G^o$. There is a one-to-one correspondence between the central extension of $\pi_0 (G) $ by $G^o$ and the second cohomology group $H^2(\pi_0(G); G^o)$ \cite[IV.3.12]{brown1982cohomology}, so
$$1 \rightarrow G^o \rightarrow G \rightarrow \pi_0 (G) \rightarrow 1$$
is characterized by $H^2(\pi_0(G); G^o) = H^2(\pi_0(G); \C^2)$. We thus show that this extension is trivial by showing that $H^2(\pi_0(G); \C^2) = 0$. 

Since $G$ is an algebraic group, its identity component $G^o$ has finite index in $G$, so $G/G^o \cong \pi_0 (G)$ is finite. For a finite group $G$, and for any finitely generated $\Z G$-module $M$, the cohomology groups $H_n (G;M)$ are finite abelian for all $n \geq 1$ \cite[4.3]{Webb}. $\Z$ is the trivial $\Z G$-module (with G acting trivially) and also finitely generated, and since $\pi_0 (G)$ is finite, we have that
$$H_n (\pi_0 (G); \Z)$$
is finite abelian for all $n \geq 1$. 

We now turn back to $H^2 (\pi_0 (G); \C^2)$. Recall that the Universal Coefficient Theorem for cohomology states that for any abelian group $A$ and a CW complex  $X$, there is a short exact sequence:
$$0 \to \mathrm{Ext}^1(H_{n-1}(X; \mathbb{Z}), A) \to H^n(X; A) \to \mathrm{Hom}(H_n(X; \mathbb{Z}), A) \to 0.$$

In our context, we consider the group $\pi_0(G)$ and the coefficients $\mathbb{C}^2$. The theorem states that there is an exact sequence:
$$0 \to \mathrm{Ext}(H_1(\pi_0(G); \mathbb{Z}), \mathbb{C}^2) \to H^2(\pi_0(G); \mathbb{C}^2) \to \mathrm{Hom}(H_2(\pi_0(G); \mathbb{Z}), \mathbb{C}^2) \to 0.$$
We analyze each term in this exact sequence in part:
\begin{enumerate}
    \item $\mathrm{Hom}(H_2(\pi_0(G); \mathbb{Z}), \mathbb{C}^2)$: Recall that $H_2(\pi_0(G); \mathbb{Z})$ is finite abelian. Since any homomorphism from 
    a finite group to $\C^2$ is the trivial because $\C^2$ is torsion-free and of infinite order, we have
    $$\mathrm{Hom}(H_2(\pi_0(G); \mathbb{Z}), \mathbb{C}^2) = 0.$$
    \item $\mathrm{Ext}(H_1(\pi_0(G); \mathbb{Z}), \mathbb{C}^2)$: We use the following two properties for calculating $\mathrm{Ext}(H, G)$, where $H$ is finitely generated and $G$ is any group:
        \begin{enumerate}
        \item $\mathrm{Ext}(H \oplus H', G) = \mathrm{Ext}(H, G) \oplus \mathrm{Ext}(H', G)$
        \item $\mathrm{Ext}(\Z / n \Z, G) = G/nG$.
        \end{enumerate}
        Since $H_1(\pi_0(G); \Z)$ is finite abelian, by the fundamental theorem of finite abelian groups, it can be expressed as the direct sum of cyclic subgroups of prime-power order. Thus, 
        $$H_1(\pi_0(G); \mathbb{Z}) = C_{p_1} \oplus C_{p_2} \oplus \cdots \oplus C_{p_n}.$$
        Every finite cyclic group of order $p_i$ is isomorphic to $\Z/p_i \Z$, so 
        $$H_1(\pi_0(G); \mathbb{Z}) = \Z/p_1 \Z \ \oplus \ \Z/p_2 \Z \ \oplus \ \cdots \oplus \ \Z/p_n \Z.$$
        Thus, by properties (a) and (b), we have 
        $$\hspace{0.3in} \mathrm{Ext}(H_1(\pi_0(G); \mathbb{Z}), \mathbb{C}^2) = \C^2/p_1 \C^2 \ \oplus \ \C^2/p_2 \C^2 \ \oplus \ \cdots \oplus \ \C^2/p_n \C^2.$$
        Since $\C^2/n\C^2$ is trivial for any $n \in \Z, n \not = 0$, we finally have that 
        $$\mathrm{Ext}(H_1(\pi_0(G); \mathbb{Z}), \mathbb{C}^2) = 0.$$
\end{enumerate}
Thus, the short exact sequence becomes 
$$0 \to 0 \to H^2(\pi_0(G); \mathbb{C}^2) \to 0 \to 0,$$
so $H^2(\pi_0(G); \mathbb{C}^2) = 0$. Since cohomology and homology are dual in this context, we use the fact that the dual space of a trivial cohomology group is a trivial homology group. Therefore, $H_2(\pi_0(G); \mathbb{C}^2) = 0.$

Since the extension is trivial, we have that i.e. $G \cong G^o \times \pi_0 (G).$ Thus, $G^o \cong G / \pi_0 (G).$ Since we assumed that there exists a homomorphism $\Gamma \rightarrow G$, $G^o$ as a quotient of $G$ admits a non-trivial homomorphism from $\Gamma$
$$\Gamma \rightarrow G^o \cong G/\pi_0 (G),$$
which is a contradiction.
\end{proof}

Theorem \ref{thm:main}(2) gives us a condition for determining when a hyperbolic homology sphere has quotient dimension 3, but the existence of such homology spheres is not yet evident. Below, we provide an example of a 3-dimensional hyperbolic homology spheres of quotient dimension three with proof. \\

\noindent (1) \noindent Let $|n| \geq 2$. Consider $1/n$ Dehn's surgery on the $4_1$ knot:
$$\pi_1( S_{1/n}^3(4_1) )= \langle \, a, b \, \, | \, \, ab^3ab^{-1}a^{-2}b^{-1}, \ ab\cdot(b^{-1}a^{-1}b^{-1}aba^{-1}ba)^n \, \rangle$$
\noindent Let us first consider the fundamental group of the knot complement 
$$ \pi_1 (S^3 - 4_1) = \langle \, a, b \, \, | \, \, ab^3ab^{-1}a^{-2}b^{-1} \, \rangle.$$ 
Since $ \pi_1 (S^3 - 4_1)$ does not map onto $A_5$ (checked below by GAP \cite{GAP}),\\
\begin{minipage}{0.75\textwidth}
\begin{figure}[H]
\centering
\includegraphics[scale=0.45]{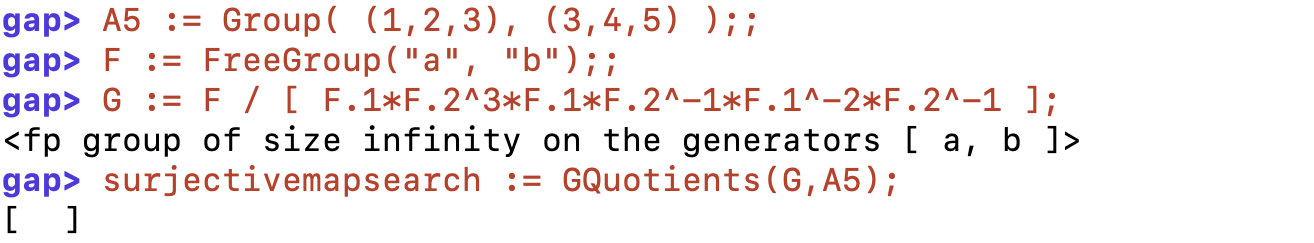}
\end{figure}
\noindent 
\end{minipage}
\begin{minipage}{0.25\textwidth}
\begin{figure}[H]
\centering
\vspace{-.15in}
\includegraphics[scale=0.05]{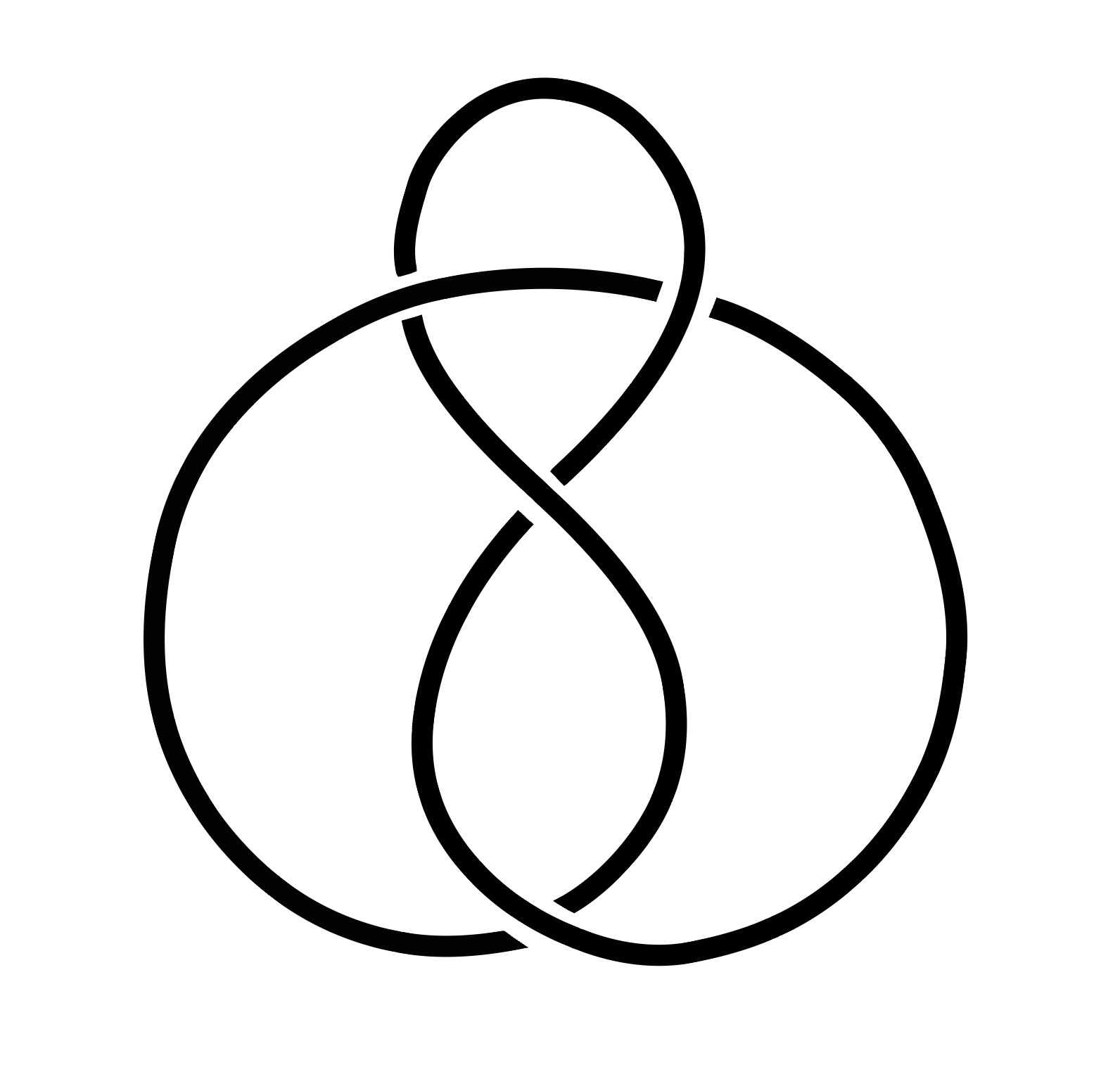}
\vspace{.05in}
\end{figure}
\end{minipage}
the fundamental group of $1/n$ surgery on $S^3 - 4_1$ cannot map onto $A_5$ either. Thus, $\pi_1(S_{1/n}^3(4_1)) \not \twoheadrightarrow A_5$, and so, by Theorem \ref{thm:main}(2), for $|n| \geq 2$, $S_{1/n}^3(4_1)$ has quotient dimension 3. We have thus found our first example of a 3-dimensional hyperbolic manifold with quotient dimension 3.

Continuing in this manner, we find that there are quite a few 3-dimensional hyperbolic manifolds with quotient dimension 3. Indeed, the table below shows a list of knots $K$ with crossing number 6-9 such that $S_{1/n}^3(K)$ has quotient dimension 3 (for integers $n$ such that the filling is hyperbolic).
 
\begin{multicols}{2}
\begin{table}[H]
\begin{tabular}{|l|l|}
\hline
Name  & Two-Bridge Notation \\ \hline
$6_1$ & $[9, 4]$            \\ \hline
$6_2$ & $[11, 3]$           \\ \hline
$7_2$ & $[11, 5]$           \\ \hline
$7_3$ & $[13, 3]$           \\ \hline
$7_5$ & $[17, 5]$           \\ \hline
$7_6$ & $[19, 7]$           \\ \hline
$7_7$ & $[21, 8]$           \\ \hline
$8_1$ & $[13, 6]$           \\ \hline
$8_2$ & $[17, 3]$           \\ \hline
$8_3$ & $[17, 4]$           \\ \hline
$8_4$ & $[19, 14]$           \\ \hline
$8_9$ & $[13, 6]$           \\ \hline
$8_{11}$ & $[17, 3]$           \\ \hline
$8_{12}$ & $[17, 4]$           \\ \hline
$8_{13}$ & $[19, 14]$           \\ \hline
$8_{14}$ & $[31, 12]$           \\ \hline
\end{tabular}
\end{table}
\begin{table}[H]
\begin{tabular}{|l|l|}
\hline
Name  & Two-Bridge Notation \\ \hline
$9_2$ & $[15, 7]$           \\ \hline
$9_3$ & $[19, 3]$           \\ \hline
$9_4$ & $[21, 5]$           \\ \hline
$9_7$ & $[29, 9]$           \\ \hline
$9_9$ & $[31, 7]$           \\ \hline
$9_{10}$ & $[33, 23]$           \\ \hline
$9_{11}$ & $[33, 14]$           \\ \hline
$9_{12}$ & $[35, 13]$           \\ \hline
$9_{14}$ & $[37, 8]$           \\ \hline
$9_{15}$ & $[39, 16]$           \\ \hline
$9_{17}$ & $[39, 14]$           \\ \hline
$9_{18}$ & $[41, 17]$           \\ \hline
$9_{19}$ & $[41, 16]$           \\ \hline
$9_{21}$ & $[43, 12]$           \\ \hline
$9_{26}$ & $[47, 18]$           \\ \hline
$9_{27}$ & $[49, 19]$           \\ \hline
\end{tabular}
\end{table}
\end{multicols}

The examples shown above of infinite families of integral homology spheres with quotient dimension 3, alongside Theorem \ref{thm:main}(1), conclude the proof of the second part of Theorem \ref{thm:main}.


\section{Open Questions}
Although the examples above prove that there exists 3-dimensional hyperbolic manifolds with quotient dimension 3, the natural next question remains: what is more common, quotient dimension 2 or 3? For $\Gamma = \pi_1(Y)$, this reduces to the following questions:
\begin{itemize}
\item How often does one have a surjection from $\Gamma$ onto $A_5$?
\item If such a surjection exists, how often does it imply $q_{\Gamma} = 2$?
\end{itemize}
While the notion of quotient dimension is itself restrictive — it only tells us about the distribution of regular finite coverings — more information about computing its value for hyperbolic 3-manifolds would improve our understanding of the symmetries of hyperbolic 3-manifolds. 


\section{Code} 

\noindent See \href{https://github.com/ms6166/quotient-dimension.git}{https://github.com/ms6166/quotient-dimension.git} for the full code.


\section*{Acknowledgements}
This paper is based on a section of my undergraduate senior thesis at Columbia University. I would like to thank my advisor, Professor Francesco Lin, for his invaluable guidance, support, and patience. The research for this project was started during the 2023 Columbia Math REU, and it was supported by the Columbia Science Research Fellowship.


\bibliographystyle{plain}
\bibliography{finalmanuscript}

\begin{thebibliography}{10}

\bibitem{Baldwin_2018}
John Baldwin and Steven Sivek.
\newblock Stein fillings and $\text{SU}(2)$ representations.
\newblock {\em Geometry Topology}, 22(7):4307–4380, December 2018.

\bibitem{zbMATH01466287}
Kurt Behnke and Oswald Riemenschneider.
\newblock Quotient surface singularities and their deformations.
\newblock In {\em Singularity theory. Proceedings of the symposium, Trieste,
  Italy, August 19-- September 6, 1991}, pages 1--54. Singapore: World
  Scientific, 1995.

\bibitem{brown1982cohomology}
K.S. Brown.
\newblock {\em Cohomology of Groups}.
\newblock Graduate Texts in Mathematics. Springer, 1982.

\bibitem{Conrad}
Keith Conrad.
\newblock Notes on semidirect products.

\bibitem{Cornwell_2016}
Christopher Cornwell, Lenhard Ng, and Steven Sivek.
\newblock Obstructions to lagrangian concordance.
\newblock {\em Algebraic; Geometric Topology}, 16(2):797–824, April 2016.

\bibitem{coxeter2013generators}
H.S.M. Coxeter and W.O.J. Moser.
\newblock {\em Generators and Relations for Discrete Groups}.
\newblock Ergebnisse der Mathematik und ihrer Grenzgebiete. 2. Folge. Springer
  Berlin Heidelberg, 2013.

\bibitem{Fintushel1990InstantonHO}
Ronald Fintushel and Ronald~J. Stern.
\newblock Instanton homology of seifert fibred homology three spheres.
\newblock {\em Proceedings of The London Mathematical Society}, pages 109--137,
  1990.

\bibitem{GAP}
The GAP~Group.
\newblock {\em {GAP -- Groups, Algorithms, and Programming, Version 4.14.0}},
  2024.

\bibitem{hatcher2002algebraic}
A.~Hatcher.
\newblock {\em Algebraic Topology}.
\newblock Algebraic Topology. Cambridge University Press, 2002.

\bibitem{quaternion}
G.~C. Joshi.
\newblock {\em Quaternions and Octonions in Nature}, pages 193--207.

\bibitem{Kirby}
Rob Kirby.
\newblock Problems in low-dimensional topology.
\newblock 1995.

\bibitem{kronheimer2003}
P.~B. Kronheimer and T.~S. Mrowka.
\newblock Dehn surgery, the fundamental group and $\text{SU}(2)$, 2003.

\bibitem{Larsen}
Michael Larsen.
\newblock How often is 84(g-1) achieved?
\newblock {\em Israel Journal of Mathematics}, 126(1):1--16, 2001.

\bibitem{Macdonald_Carter_MacDonald_Segal_1995}
I.~G. Macdonald, Roger~W. Carter, Ian~G. MacDonald, and Graeme~B. Segal.
\newblock {\em Linear Algebraic Groups}, page 133–185.
\newblock London Mathematical Society Student Texts. Cambridge University
  Press, 1995.

\bibitem{perelman2002entropyformularicciflow}
Grisha Perelman.
\newblock The entropy formula for the ricci flow and its geometric
  applications, 2002.

\bibitem{perelman2003finiteextinctiontimesolutions}
Grisha Perelman.
\newblock Finite extinction time for the solutions to the ricci flow on certain
  three-manifolds, 2003.

\bibitem{perelman2003ricciflowsurgerythreemanifolds}
Grisha Perelman.
\newblock Ricci flow with surgery on three-manifolds, 2003.

\bibitem{10.1112/plms/s3-24.2.217}
Robert Riley.
\newblock Parabolic representations of knot groups, i.
\newblock {\em Proceedings of the London Mathematical Society},
  s3-24(2):217--242, 03 1972.

\bibitem{CJPS_1998__16_}
Jean-Pierre Serre.
\newblock {\em Moursund {Lectures}}.
\newblock Number~16 in Cours de Jean-Pierre Serre. 1998.

\bibitem{Sivek_2021}
Steven Sivek and Raphael Zentner.
\newblock A menagerie of $\text{SU}(2)$-cyclic 3-manifolds.
\newblock {\em International Mathematics Research Notices},
  2022(11):8038–8085, January 2021.

\bibitem{Thurston+1997}
William~P. Thurston.
\newblock {\em Three-Dimensional Geometry and Topology, Volume 1}.
\newblock Princeton University Press, Princeton, 1997.

\bibitem{Webb}
Peter~J Webb.
\newblock An introduction to the cohomology of groups.

\bibitem{Zentner_2017}
Raphael Zentner.
\newblock A class of knots with simple $\text{SU}(2)$-representations.
\newblock {\em Selecta Mathematica}, 23(3):2219–2242, February 2017.

\bibitem{Zentner_2018}
Raphael Zentner.
\newblock Integer homology 3-spheres admit irreducible representations in
  $\text{SL}(2,c)$.
\newblock {\em Duke Mathematical Journal}, 167(9), June 2018.

\end{thebibliography}

\end{document}